\documentclass[12pt]{article}

\newif\ifarXiv
\arXivtrue

\ifarXiv
\usepackage{e-jc_new}
\else
\usepackage{e-jc}
\fi

\usepackage{amsthm,amsmath,amssymb}

\usepackage[T1]{fontenc}
\usepackage[english]{babel}
\usepackage{mathtools}
\usepackage{rotating}

\usepackage{graphicx}

\usepackage[colorlinks=true,citecolor=black,linkcolor=black,urlcolor=blue]{hyperref}

\theoremstyle{plain}
\newtheorem{theorem}{Theorem}[section]

\newtheorem{lemma}[theorem]{Lemma}

\newtheorem{corollary}[theorem]{Corollary}

\theoremstyle{definition}
\newtheorem{definition}[theorem]{Definition}

\newtheorem{example}[theorem]{Example}

\newtheorem{problem}[theorem]{Problem}

\theoremstyle{remark}

\usepackage{amssymb}
\usepackage{amsmath}

\newcommand{\Gv}[1]{(G,\{v_#1\})}
\newcommand{\Gr}[1]{(K_r,\{v_#1\})}
\newcommand{\Gk}[2]{(#1,#2)}
\newcommand{\Gg}[2]{(#1,\{v_#2\})}

\newcommand{\cp}[2]{#1 ~\!\square~\! #2}
\newcommand{\st}[2]{#1 \boxtimes~\!\! #2}

\author{Tomer Kotek\thanks{Partially supported by the Fein foundation, the graduate school of the Technion, the Austrian National Research Network S11403-N23 (RiSE) of the Austrian Science Fund (FWF), and by the Vienna Science and Technology
Fund (WWTF) grant PROSEED.}\\
\small Faculty of Informatics,\\
\small Vienna University of Technology,\\
\small Favoritenstr. 9-11, A-1040 Vienna, Austria\\
\small\tt kotek@forsyte.at \\
\and
James Preen\\
\small Mathematics\\[-0.8ex]
\small Cape Breton University\\[-0.8ex]
\small Sydney, NS B1P6L2, Canada\\
\small\tt james\_preen@capebretonu.ca\\
\and
Peter Tittmann \\
\small Faculty Mathematics / Sciences / Computer Sciences \\[-0.8ex]
\small Hochschule Mittweida - University of Applied Sciences \\[-0.8ex]
\small Mittweida, Germany\\
\small\tt peter@hs-mittweida.de\\
}

\date{\dateline{Jan 1, 2012}{Jan 2, 2012}}
\begin{document}
\title{\bf Domination Polynomials of Graph Products}
\maketitle

\abstract{
The domination polynomials of binary graph operations, aside from union, join and corona,
have not been widely studied. We compute and prove recurrence formulae and properties of the domination polynomials 
of families of graphs obtained by various products, ranging from explicit formulae and recurrences for specific families to more general results. As an application, we show the domination polynomial is computationally hard to evaluate. 
}

\section{Introduction and Defintions}

This paper discusses simple undirected graphs $G=(V,E)$. 
A vertex subset $W \subseteq V$ of $G$ is a \emph{dominating
  set} in $G$, if for each vertex $v \in V$ of $G$ either
$v$ itself or an adjacent vertex is in $W$. 
\begin{definition}
Let $G = (V, E)$ be a graph.
The
\emph{domination polynomial} $D(G, x)$ is 
given by 
\[
 D(G,x)=\sum_{i=0}^{|V|} d_i(G) x^i\,,
\]
where $d_i(G)$ is the number of dominating sets of size $i$ in $G$. 
The \emph{domination number} of a graph $G$, denoted $\gamma(G)$, is the smallest $i$ such that $d_i(G)>0$.
\end{definition}

In \cite{ar:KPT1} we showed that there exist recurrence relations for the domination polynomial which allow for efficient schemes to compute the polynomial for some types of graphs. A recurrence for the domination polynomial of the {\em path graph} with $n$ vertices ($P_n$) was shown in \cite{ar:AlikhaniPengCertainII2010} to be 
\begin{equation}\label{e:path}
D(P_{n+1},x) = x( D(P_{n},x) + D(P_{n-1},x) +D(P_{n-2},x) )
\end{equation}
where $D(P_0,x) = 1$, $D(P_1,x) = x$ and $D(P_2,x) = x^2 + 2x$. 
Note that the complete graphs $K_j \cong P_j$ for $0\leq j\leq2$ 
and that $D(K_r,x) = (x+1)^r-1$.

Given any two graphs $G$ and $H$ we define the {\em Cartesian product}, denoted $\cp{G}{H}$, to be 
the graph with vertex set $V(G) \times V(H)$ and edges between two vertices $(u_1,v_1)$ and $(u_2,v_2)$ if and only
if either $u_1=u_2$ and $v_1v_2 \in E(H)$ or $u_1u_2 \in E(G)$ and $v_1=v_2$. 
For $S\subseteq V(G)$ and $T \subseteq V(H)$ let $(S,T)$ be the subgraph of $\cp{G}{H}$ containing all
vertices $(u,v)$ such that $u \in S$ and $v \in T$. 

This product is well known to be commutative and, if $G$ 
is a disconnected graph with components $G_1$ and $G_2$, then $\cp{G}{H} = (\cp{G_1}{H}) \cup (\cp{G_2}{H})$, so that 
\[
D(\cp{G}{H},x) = D(\cp{G_1}{H},x)  D(\cp{G_2}{H},x).
\]
 Despite these properties, it is difficult to determine much about this product, even in such simple cases as the grid graphs $\cp{P_n}{P_m}$, especially in the case of dominating sets.
The strong product ($\st{G}{H}$) is the graph which is formed by taking the graph $\cp{G}{H}$ and then
additionally adding edges between vertices $(u_1,v_1)$ and $(u_2,v_2)$ if both $u_1u_2 \in E(G)$ and $v_1v_2 \in E(H)$.

The domination numbers of graph products have been extensively studied in the literature, see e.g \cite{ar:Faudree1990,ar:Klavzar1995,ar:Aharoni2009,ar:Benecke2010,ar:Clark00,ar:Gravier1995,ar:Sun2004,ar:Gravier1997,ar:JacobsonKinch83,Liu2010,ar:Mollard2012,ar:Clark12}. 
In particular, a large number of papers have addressed the domination number of Cartesian products, inspired by the conjecture by V. G. Vizing \cite{ar:Vizing1968} that $\gamma(\cp{G}{H}) \geq \gamma(G)\times\gamma(H)$ (see \cite{ar:VizingSurvey2012} for a recent survey.)
In contrast, although the domination polynomial has been actively studied in recent years, almost no attention has been given to the domination polynomials of graph products.

 The \emph{closed neighborhood}
$N_G[W]$ of a vertex set $W$ in $G$ contains $W$ and all vertices adjacent to vertices in $W$. When $W=\{v\}$ we will write
$N_G[v]$ or just $N[v]$ if the graph we are working in is obvious. We define $N_G(W)$ as the \emph {open neighbourhood}
which includes all neighbours of $W$ that are not in $W$, so that $N_G(W) := N_G[W]\setminus W$.
If $S$ is a set of vertices from $G$ we use $G-S$ to mean the graph resulting from the deletion of all vertices in $S$ from $G$, 
and let $G-v$ be $G-\{v\}$.  
The {\em vertex contraction} $G/v$ denotes the graph obtained from $G$ by the removal of $v$ and 
the addition of edges between any pair of non-adjacent neighbors of $v$. 

An outline of the paper is as follows. In section \ref{ss:gk2} a decomposition formulae for an arbitrary graph's Cartesian product with $K_2$  is given and also its strong product with $K_r$ for any $r\geq 2$.
In Section \ref{se:kk} we give exact formulae for families of Cartesian products of complete graphs.
Section \ref{se:pk2} gives a recurrence relation for any graph family which contains $\cp{P_n}{K_2}$ that 
uses only six smaller graphs.
A generalisation of the result in section \ref{se:pk2} is given in section \ref{se:pnkr}, where we give a recurrence for $\cp{P_n}{K_r}$. In section \ref{se:seq}, we discuss why recurrence relations can be deduced to exist for many graph products and show implications of their existence to properties of sequences of coefficients of the domination polynomial. Finally, we use a result 
from the paper to show the Turing hardness of the domination polynomial.

\section{Domination polynomials of products with Complete Graphs}\label{ss:gk2}

Let us suppose that $V(K_2):=\{v_1, v_2\}$ in the product $\cp{G}{K_2}$ and let $G$ be any non-null graph.
We will concentrate first on the vertices in $\Gv{1}$:
every vertex subset $W$ of $\Gv{1}$ can be a subset of some dominating set $S$ in $\cp{G}{K_2}$ so long as some vertices in $\Gv{2}$ are included in $S$ as described below.
Let $W\subseteq V(G)$, so, by definition, $\Gg{W}{1}$ dominates the vertices in $\Gg{N_G[W]}{1}$ 
and those in $\Gg{W}{2}$. 
For $S$ to dominate, all other vertices $(y,v_1)$ must then be dominated by $(y,v_2)$, their only neighbor outside of $\Gv{1}$.

\begin{theorem}\label{t:GK2}
The domination polynomial for 
$D(\cp{G}{K_2},x)$ is equal to:
\[
\frac{x^{|V(G)|}}{x+1}\times \sum_{W \subseteq V(G)}   \frac{  \left( 
D( J_W/z,x ) + D( J_W-N_{J_W}[z],x ) + D( J_W,x) - D( J_W-z,x ) \right)
}{x^{|N_G(W)|}}
\]
where $J_W$ is the subgraph of $G$ induced by $N_G[W]$ with a new vertex
$z$ joined to the union of $W$ and $N(V(G)-N_G[W])$.
\end{theorem}

\begin{proof}
Suppose that $W\subseteq V(G)$, so that, as above, we know that 
in any dominating set for $\cp{G}{K_2}$ if the
only vertices from $\Gv{1}$ are $W$ then we must also include 
$\Gk{V(G)-N_G[W]}{v_2}$ . In this way all vertices
in $\Gv{1}$ are dominated 
by these 
\[
|W| + |V(G)\setminus N_G[W]| = |W| + |V(G)|-|N_G[W]| =  |V(G)|-|N_G(W)|
\]
 vertices, giving the powers of $x$ as in the theorem. 

It now remains to ensure that all of the vertices in 
$\Gv{2}$ are dominated. Using the vertices forced
to dominate $\Gv{1}$ we see that, in $\Gv{2}$,
 every vertex in either $W$ or in $N[V(G)-N[W]]$ is dominated.
The only vertices not dominated are therefore those 
which are not in $W$ but have no neighbours outside of $N[W]$.
Let us call this set $T_W$.

We now introduce the graph $J_W$ which is formed by taking the
subgraph of $G$ induced by $N[W]$ and adding a new vertex $z$
which is adjacent to every vertex either in $W$ or $N(V(G)-N[W])$.
The vertices which $z$ is joined to are exactly those {\em not} in $T_W$.
Thus we want to count all sets of vertices in $J_W \setminus \{z\}$ such
that $T_W$ is dominated.

From \cite{ar:KPT1}, $p_z(J_W,x)$ generates the dominating sets
for $J_W - N[z]$ which additionally dominate the vertices of $N(z)$.
Each of these sets when combined with $z$ is a dominating set for $J_W$ in
which $T_W$ is dominated and $z$ is only dominated by itself.
All other sets which dominate $T_W$ must then include a vertex from $N(z)$ and hence
they will be a dominating set for both $J_W$ and $J_W -z$. 
The difference of domination polynomials $D(J_W,x) - D(J_W - z,x)$ generates all such
sets which include $z$ and so $p_z(J_W,x) + D(J_W,x) - D(J_W-z,x)$ generates
all sets of vertices in $J_W$ that dominate $T_W$ and include $z$.

Since $z$ is not adjacent to any vertex of $T_W$ 
the generating function counting all sets of vertices in $J_W\setminus \{z\}$ such
that $T_W$ is dominated satisfies the following relation, using Theorem
2.1 of \cite{ar:KPT1}  for the expansion of $p_z(J_W,x)$: 
\begin{eqnarray*}
 &&\frac{p_z(J_W,x) + D(J_W,x) - D(J_W-z,x)}{x} \\
&=&\frac{ x D(J_W/z,x) + x D(J_W-N_{J_w}[z],x) + D(J_W-z,x) -D(J_W,x)}{x(x+1)}
\\ && ~+~ \frac{ D(J_W,x) - D(J_W-z,x)}{x} \\
&=&
\frac{ D(J_W/z,x) +D(J_W-N_{J_W}[z],x) }{x+1}
 + \frac{ D(J_W,x) - D(J_W-z,x)}{x+1} \\
\end{eqnarray*}
Putting this together with our first observation finishes the proof.
\end{proof}

The following result was also proven independently in \cite{ar:brown} as their Lemma 3:

\begin{theorem}\label{t:spgkr}
For any graph $G$ 
\[
D(\st{G}{K_r},x) = D(G,(x+1)^r-1)
\]
\end{theorem}

\begin{proof}
Let $u$ be a vertex of $G$ and $v \in V(K_r)$; the closed neighborhood of the vertex $(u,v)$ is $(N_G[u],K_r)$.
For any $X \subseteq V(G)$, let $\{ A_x ~|~ x\in X \}$ be a family of arbitrary non-empty subsets of $V(K_r)$.
We then have that such a set $X$ is a dominating set of $G$ if and only if 
\[
\bigcup_{x\in X} \{ (x,v) ~|~ v \in A_x \}
\]
is a dominating set of $\st{G}{K_r}$. Consequently, each vertex $u$ of a dominating set of $G$ corresponds to all
non-empty subsets of $(u,K_r)$ in $\st{G}{K_r}$, which are counted by the generating function $(x+1)^r -1$.
\end{proof}

Theorem \ref{t:spgkr} can be used to generalise recurrence relations for the domination polynomial of any families of graphs, such as for $H_{n,r} := \st{P_n}{K_r}$ as follows:

\begin{corollary}
For any integers $n\geq 3$  and $r\geq 1$ we have
\[
D(H_{n+1,r},x) = ((x+1)^r-1) \left( D(H_{n,r},x) +  D(H_{n-1,r},x) + D(H_{n-2,r},x)  \right)
\]
\end{corollary}

\begin{proof}
From equation (\ref{e:path})  
and using Theorem \ref{t:spgkr}
we have 
\begin{eqnarray*}
D(H_{n+1,r},x) &=& D(\st{P_{n+1}}{K_r},x) \\
&=& D(P_{n+1},(x+1)^r-1) \\
&=& ((x+1)^r-1) ( D(P_{n},(x+1)^r-1) + D(P_{n-1},(x+1)^r-1) \\ &&~~+D(P_{n-2},(x+1)^r-1) ) \\
&=& ((x+1)^r-1) ( D(H_{n,r},x) + D(H_{n-1,r},x) +D(H_{n-2,r},x) )
\end{eqnarray*}
as required.
\end{proof}

Note that, as shown in \cite{ar:AlikhaniPengCertainII2010}, the same recurrence as equation (\ref{e:path}) holds for the cycle graphs $C_n$ hence there will be an identical generalisation for the domination polynomial of $\st{C_n}{K_r}$.

\begin{corollary}\label{c:hn}
For any integers $n>3$  and $r\geq 1$ we have $D(\st{C_{n+1}}{K_r},x) = $
\[
((x+1)^r-1) \left( D(\st{C_n}{K_r},x) +  D(\st{C_{n-1}}{K_r},x) + D(\st{C_{n-2}}{K_r},x)  \right)
\]
\end{corollary}

\section{The Domination polynomial of $\cp{K_r}{K_s}$}\label{se:kk}

\begin{theorem}\label{t:krk2}
The domination polynomial for $\cp{K_r}{K_2}$ is
\[
(x+1)^{2r} - 2(x+1)^r + 2x^r +1 = ((x+1)^r - 1)^2 +2x^r
\]
\end{theorem}

\begin{proof}
Let us again suppose that $V(K_2):=\{v_1, v_2\}$ in the product $\cp{K_r}{K_2}$.
All  sets which contain only vertices from some $\Gr{j}$ are generated by $2((x+1)^r-1) + 1$ where the first term utilises the
non-empty sets and the last term comes from the empty set. Amongst these sets, only those which contain all $r$ vertices from $\Gk{K_r}{v_j}$
(for $j=1$ or 2) are dominating in $\cp{K_r}{K_2}$, giving the term $2x^r$. 
All vertex subsets of $\cp{K_r}{K_2}$ are counted by $(x+1)^{2r}$, so the domination polynomial for it  is 
\[
(x+1)^{2r} - (2(x+1)^r - 1) + 2x^r = ((x+1)^r - 1)^2 +2x^r \qedhere
\]
\end{proof}

Note that Theorem \ref{t:krk2} can also be deduced from Theorem \ref{t:GK2}, although it is a more involved calculation, even after using the symmetry inherent when $G = K_r$. Theorem \ref{t:krk2} can be generalised in the following way:

\begin{theorem}\label{t:krks}
The domination polynomial for $\cp{K_r}{K_s}$ is
\[
D(\cp{K_r}{K_s},x) = \left( (x+1)^r -1 \right)^s - \sum_{k=1}^{s-1} {s \choose k} (-1)^k  \left( (x+1)^{s-k} -1 \right)^r 
\]
\end{theorem}

\begin{proof}
We can imagine the vertices of $\cp{K_r}{K_s}$ as elements of an $r\times s$ matrix; for a dominating set in this graph we need
to have at least one element in every row and column. The simplest way this can be achieved is to have at least one vertex in every column and the ordinary generating function that generates such sets is $\left( (x+1)^r -1 \right)^s$.
However, it is also possible to have empty sets in some columns, so long as each row contains at least one element:

There are $s$ choices for the case of one empty column and, given that choice, the generating function counting non-empty rows of $s-1$ elements is $\left( (x+1)^{s-1} -1 \right)^r$. However, some of the sets counted in this way will have more than one empty column; by the principle of inclusion-exclusion, we now need to subtract the $s \choose 2$ ways to choose a pair of columns to be empty.

The polynomial counting dominating sets with at least two columns empty is 
\[
\left( (x+1)^{s-2} -1 \right)^r
\]
 but this then includes sets with more than two empty columns and so the inclusion-exclusion process will continue.
The final case will be when we have all but one column empty, in which case the only possible dominating set contains all $r$ vertices from one column. The term counting all such sets will be $sx^r = {s \choose s-1} ((x+1) -1)^r$, which matches the term in the sum in the theorem when $k=s-1$. Combining all of these cases together completes the proof.
\end{proof}

\begin{corollary}\label{t:krk3}
The domination polynomial for $\cp{K_r}{K_3}$ is
\[
((x+1)^r -1)^3+3x^r((x+2)^r-1)
\]
\end{corollary}

\begin{proof}
Substituting $s=3$ into Theorem \ref{t:krks} we get
\begin{eqnarray*}
D(\cp{K_r}{K_3},x) &=& \left( (x+1)^r -1 \right)^3 - \sum_{k=1}^{2} {3 \choose k} (-1)^k  \left( (x+1)^{3-k} -1 \right)^r \\
&=& \left( (x+1)^r -1 \right)^3 + 3\left(  \left( (x+1)^2 -1 \right)^r -  \left( (x+1) -1 \right)^r \right) \\
&=& \left( (x+1)^r -1 \right)^3 + 3\left(  \left( (x(x+2) \right)^r -  x^r \right) \\
&=& \left( (x+1)^r -1 \right)^3 + 3x^r\left( (x+2)^r -  1 \right) \qedhere
\end{eqnarray*}
\end{proof}

\section{The Domination Polynomial for $\cp{P_n}{K_2}$}  \label{se:pk2}

Let $L_n$ be the graph $\cp{P_n}{K_2}$
and label the vertices of the
two copies of $P_n$ as $u_1, \ldots, u_n$ and $v_1, \ldots, v_n$.
Note that the graph $L_{n-1}$ is formed
from $L_n$ by deletion of $u_n$ and $v_n$. The domination polynomials of the first six graphs in the family are given in Table \ref{tab:ln}.

\begin{center}
\begin{table}
\caption{The domination polynomials for the graphs $\cp{P_n}{K_2}$}\label{tab:ln}
~~~~~~~$
\begin{array}{c|c}
n & D(\cp{P_n}{K_2},x) \\\hline
1 
& 
{x}^{2}+2\,x
\\
2 
&
{x}^{4}+4\,{x}^{3}+6\,{x}^{2}
\\
3
&
{x}^{6}+6\,{x}^{5}+15\,{x}^{4}+16\,{x}^{3}+3\,{x}^{2}
\\
4
&
{x}^{8}+8\,{x}^{7}+28\,{x}^{6}+52\,{x}^{5}+48\,{x}^{4}+12\,{x}^{3}
\\
5
&
{x}^{10}+10\,{x}^{9}+45\,{x}^{8}+116\,{x}^{7}+178\,{x}^{6}+148\,{x}^{5}+47\,{x}^{4}+2\,{x}^{3}
\\
6
&
{x}^{12}+12\,{x}^{11}+66\,{x}^{10}+216\,{x}^{9}+453\,{x}^{8}+604\,{x}^{7}+470\,{x}^{6}+168\,{x}^{5}+17\,{x}^{4}\\\hline
\end{array}
$
\end{table}
\end{center}

We first prove a small result which will be used in the main theorem of this section.

\begin{lemma}\label{l:both}
The polynomial $A_{n} (x)$ counting the dominating sets of $L_n$ such
that both $u_n$ and $v_n$ are included satisfies
the following:
\[
A_n (x) := x^2 \left( D(L_{n-1},x) + D(L_{n-2},x) - A_{n-2}(x) \right)
\]
\end{lemma}

\begin{proof}
Every dominating set for either $L_{n-1}$ or $L_{n-2}$ will be a dominating set for $L_n$ when combined with $u_n$ and $v_n$
since these two vertices dominate themselves and their neighbours. Any set $S$ which is a dominating set in both 
$L_{n-1}$ and $L_{n-2}$ cannot contain either $u_{n-1}$ or $v_{n-1}$ since they are not in $L_{n-2}$ and 
hence $S$ must contain both $u_{n-2}$ and $v_{n-2}$ in order for the former pair of vertices to be dominated.
Thus exactly $x^2 A_{n-2} (x)$ sets  are counted twice and this is subtracted to give our result.
\end{proof}

\begin{theorem}\label{t:lad}
The dominating polynomial for $L_n$ satisfies the recurrence:
\begin{eqnarray*}
D(L_n ,x) &=& x(x+2) D(L_{n-1},x) 
+  x(x+1) D(L_{n-2},x)  \\
&&+  x^2(x+1) D(L_{n-3},x) 
-  x^3 D(L_{n-4},x) 
-  x^3 D(L_{n-5},x) 
\end{eqnarray*}
\end{theorem}

\begin{proof}
Let $T$ be a dominating set for $L_n$ and set
$T_1:= T / \{u_n, v_n\}$.  If $T_1 = T$ then (in order to have $u_n$
and $v_n$ dominated) we can 
conclude that $|T \cap \{ u_{n-1}, v_{n-1} \} | = 2$ and the polynomial counting such sets will be $A_{n-1} (x)$  as in Lemma 
\ref {l:both}.
This gives us the contribution $x^2 \left(  D(L_{n-2},x) +  D(L_{n-3},x) - A_{n-3}(x) \right)$ for our summation.

Now suppose $|T \cap \{u_n,v_n\}|\geq 1$; if $T_1$ is a 
dominating set for $L_{n-1}$ then $T$ will be a dominating set for
$L_n$.  Thus we get the term $x(x+2)D(L_{n-1},x)$, 
the $x(x+2)$  coming from that we can use $u_n$ and/or $v_n$ 
with $T_1$ to form a dominating set.

However, there are circumstances under which $T_1$ 
does not have to be a dominating set for $L_{n-1}$, since 
$u_{n-1}$ and $v_{n-1}$ in $L_{n-1}$ 
might be only dominated by $u_n$ or $v_n$ in $T$.
Let us now consider the ways that exist such that $u_{n-1}$ and $v_{n-1}$ can be 
not dominated in $T_1$ but dominated in $T$.

If both $u_{n-1}$ and $v_{n-1}$ are undominated by $T_1$
then we must have $|T \cap \{u_n,v_n\}|=2$ to dominate those vertices and 
also $|T \cap \{u_{n-3},v_{n-3}\}|=2$ to dominate $u_{n-2}$ and $v_{n-2}$, 
giving the term $x^2 A_{n-3} (x)$ which will cancel that term introduced at the start of the 
proof.

We are now left to count just the dominating sets for $L_{n-2}$ which 
include only one of $u_{n-2}$ and $v_{n-2}$. These sets will make a previously uncounted
dominating set for $L_n$ when combined with $v_n$ and/or $u_n$ respectively.
These are the four different possibilities, defining $S := T \cap \{u_n,v_n,u_{n-1},v_{n-1},u_{n-2},v_{n-2}\}$:

\begin{enumerate}
\renewcommand{\labelenumi}{(\roman{enumi})}
\item $S=\{u_n,v_n,v_{n-2}\}$
\item $S=\{u_n,v_{n-2}\}$ 
\item $S=\{v_n,u_{n-2}\}$
\item $S=\{u_n,v_n,u_{n-2}\}$
\end{enumerate}

To count these possibilities we can now consider the different 
ways that exactly one of $u_n$ or $v_n$ can be combined
with a dominating set for $L_{n-2}$ which will lead to the contribution
of the term $xD(L_{n-2},x)$ to our sum. 
Suppose $Q$ is a dominating set for $L_{n-2}$; we will split into
subcases depending on $r:= | Q\cap \{u_{n-2} , v_{n-2} \} |$ as follows:
 
Every set $Q$ satisfying $r=2$ can be converted into
a set of the type of possibility (i) (by adding $u_n$ and switching 
$v_n$
for $u_{n-2}$), but this new set will not be a dominating set for $L_n$ when
$u_{n-3}$ is solely dominated by $u_{n-2}$ in $Q$; that is when 
$ Q \cap \{u_{n-3},v_{n-3}, u_{n-4}\} = \varnothing$. Let the sets of this form which have $v_{n-4} \in Q$ be counted by the 
polynomial $J(x)$ and such sets which also do not include $v_{n-4}$ are necessarily $x^2 A_{n-5}(x)$ as in Lemma \ref{l:both}.

When $r=1$ we can add $u_n$ or $v_n$ as appropriate and have possibilities (ii) and (iii) for $S$.   
In the case when $r=0$, $Q$
must include both $u_{n-3}$ and $v_{n-3}$ to be dominating. No such set can be
combined with just one more vertex to make a dominating set for $L_n$, and we can count
the sets with $r=0$ (and one additional unspecified vertex) using the polynomial $x A_{n-3}(x)$.
Putting these terms together, we see that possibilities (i),(ii) and (iii) are
counted by 
\[
x(D(L_{n-2},x) -  J(x) - x^2 A_{n-5} -  A_{n-3}(x)).
\]

Finally, we can count the dominating sets for $L_n$ with $S$ as in possibility (iv)
by using $x^3 D(L_{n-3},x) + x J(x)$. We make a slight
adjustment in the same way as in the subcase when $r=0$ since
a set in which only $u_{n-3}$ is not dominated in $L_{n-3}$ 
will still be a dominating set in $L_n$ when combined with this $S$, and 
the polynomial counting such sets exactly matches the definition of $xJ(x)$.

Using Lemma \ref{l:both} again, we get that 

\[
  x^3 A_{n-5}(x)  +  x A_{n-3}(x)= x^3 \left( D(L_{n-4},x) + D(L_{n-5},x)  \right).
\]

and so, summing all of our terms together, we can count all possible dominating sets $T$ for $L_n$ by using the polynomial
in the statement in the theorem.
\end{proof}

Note that at no point did we either concern ourselves with the structure beyond
$u_{n-5}$ and $v_{n-5}$ or utilise the symmetry of $\cp{P_n}{K_2}$, and hence this same recurrence also holds for 
any family of graphs with $\cp{P_6}{K_2}$ as a pendant subgraph.

We can again use Theorem \ref{t:spgkr}  as in Corollary \ref{c:hn} 
to find the domination polynomial for the strong product 
$Z_{n,r} := \st{L_n}{K_r}$:

\begin{corollary}
For any integers $n\geq 6$  and $r\geq 1$ we have
\begin{eqnarray*}
D(Z_{n,r} ,x) &=&  \left( (x+1)^{2r} -1 \right) D(Z_{n-1,r},x) \\
&&+  ( (x+1)^r -1 ) (x+1)^r  D(Z_{n-2,r},x)  \\
&&+  ((x+1)^r -1)^2 (x+1)^r D(Z_{n-3,r},x) \\
&& -  ((x+1)^r -1)^3 (D(Z_{n-4,r},x) + D(Z_{n-5,r},x) )
\end{eqnarray*}
\end{corollary}

\begin{proof}
Let us substitute $y:=(x+1)^r -1$ to simplify calculations:
\begin{eqnarray*}
D(Z_{n,r} ,x) &=& D(L_n,(x+1)^r-1)  \\
&=& D(L_n,y)  \\
&=& y(y+2) D(L_{n-1},y) 
+  y(y+1) D(L_{n-2},y)  \\
&&+  y^2(y+1) D(L_{n-3},y) 
-  y^3 D(L_{n-4},y) 
-  y^3 D(L_{n-5},y)\\ 
&=&  y(y+2) D(Z_{n-1,r},x) 
+  y(y+1) D(Z_{n-2,r},x)  \\
&&+  y^2(y+1) D(Z_{n-3,r},x) 
-  y^3 D(Z_{n-4,r},x) 
-  y^3 D(Z_{n-5,r},x)
\end{eqnarray*}
Utilising now that $y+1:=(x+1)^r$ 
we get the desired result.
\end{proof}

\section{The Domination Polynomial of $\cp{P_n}{K_r}$}\label{se:pnkr}

We denote by $M_{n,r} := \cp{P_n}{K_r}$ the Cartesian product
of the path $P_n$ and the complete graph $K_r$, where $n$ and $r$
are non-negative integers. We will utilise the linear structure of $P_n$ and 
refer to the copy of $K_r$ corresponding to a vertex of degree one 
in $P_n$ as at the {\em left end} and the copy of $K_r$ adjacent to it as the {\em second one}.
Let $m^t_{n,r}(x)$ be the domination
polynomial of $M_{n,r}$ under the condition that exactly $t$ of the
$r$ vertices of the left end $K_r$  do not necessarily need to be dominated.

Let $\delta_{t,r}:= [t=r]$ denote the Kronecker delta function.
For $n=0$  and $n=1$ the graph $M_{n,r}$ is the null graph and $K_r$ respectively and so only the case of
the empty dominating set needs to be considered carefully. For $n=2$ the case $t=0$ and $r>0$ corresponds to
Theorem \ref{t:krk2} and that proof can be generalised to give the result here.
\begin{eqnarray}
m^t_{0,r}(x) &=& 1 \label{e:pnkr} \nonumber\\
m^t_{1,r}(x) &=& (x+1)^r - 1 + \delta_{t,r}\\\nonumber
m^t_{2,r}(x) &=&  
(x+1)^{2r} - 2(x+1)^r + x^r + 1 + x^{(r-t)}(x+1)^t - \delta_{t,r}
\end{eqnarray}

From these equations we can establish the following recurrence relations for $m^t_{n,r}$ in general

\begin{theorem}\label{t:PKr}
The dominating polynomial for 
$\cp{P_n}{K_r}$ satisfies
\[
D(\cp{P_n}{K_r},x) = D( M_{n,r} ,x ) = m^0_{n,r}(x) = \sum_{t=0}^r {r \choose t} x^t m^t_{n-1,r}(x) 
\]
where $m^t_{n-1,r}(x)$ can be evaluated recursively using recurrence relations.
\end{theorem}

\begin{proof}
Suppose we have a set $S$ from $M_{n-1,r}$ which dominates all vertices in $M_{n-2,r}$ and $t$ vertices in the left end.
We first suppose that $t<r$ so that at least one vertex is undominated.

 In order to form a dominating set for $M_{n,r}$, a non-empty vertex subset of the left end $K_r$ must be added to $S$. There are $r\choose t$ ways to choose $t$ vertices and the contribution to the domination polynomial is a factor of $x^t$ in each case. These $t$ vertices will then cover $t$ vertices of the second $K_r$, so we get $m^t_{n-1,r}$ as the polynomial counting the possibilities for $S$.

It is also possible that all vertices of the left end are dominated by a set which doesn't include any vertices from its $K_r$.
In this case all $r-t$ vertices in the second $K_r$ which are adjacent to the undominated vertices of the left end must be 
added to $S$, giving a
factor of $x^{r-t}$ for the domination polynomial. This then allows the choice of an arbitrary subset of the remaining 
$t$ vertices of the second $K_r$ to make up $S$ which means that some vertices in the third $K_r$ don't need to be dominated. Putting these choices together with that from the first paragraph gives us, for $0\leq t < r$:

\begin{equation}\label{e:mtnr}
m^t_{n,r}(x) = \sum_{i=1}^r {r \choose i} x^i m^i_{n-1,r}(x) + x^{r-t} \sum_{i=0}^t {t \choose i} x^i m^{r-t+i}_{n-2,r}(x)  
\end{equation}

When $t=r$ we have all vertices of the left end already dominated, and so we are free to select an arbitrary subset of it
to create a dominating set, giving:

\[
m^r_{n,r}(x) = \sum_{i=0}^r {r \choose i} x^i m^i_{n-1,r}(x)
\]

These equations when used recursively with equations (\ref{e:pnkr}) and  (\ref{e:mtnr}), produce our result for any $n$ and $r$
since $t=0$ is the case when $M_{n,r}$ is dominated.
\end{proof}

\section{Domination polynomials of other graph products}\label{se:seq}
The purpose of this section is to place the domination polynomial in a
wider context.
So far we have considered only Cartesian products and strong products.
In this section, we investigate the domination polynomials of a
general class of graph products.
To do that, we compare the domination polynomial to the Tutte
polynomial and discuss
a common framework which contains both of these graph polynomials. The
common framework
is known to guarantee the existence of recurrence relations for
sequences of graph products.
We show that this implies that sequences coefficients of the
domination polynomial satisfy recurrence relations.
Since the results in this section apply to a wide class of graph
products, it is natural that they
given in a less exact form than results in previous sections.

In \cite{ar:NoyRibo}, M. Noy and A. Rib\'o considered the Tutte
polynomials of {\em recursively constructible} sequences of graphs.
A sequence $G_{1},G_{2},\ldots,G_{n},\ldots$ is recursively constructible
if it can be obtained from an inital graph by repeated application
of fixed elementary operations involving addition of vertices and
edges and deletion of edges. Some familiar recursively constructible
graph sequences are paths, cycles, 
stars and wheels. Noy and Rib\'o proved that for every such sequence, the
Tutte polynomials $T(G_{n},x,y)$ of the graphs in the sequence satisfy
a linear recurrence relation with coefficients in the polynomial ring
$\mathbb{Z}[x,y]$. 
In a recent paper \cite{ar:BGOP2012}, the authors disprove a conjecture of Noy and Rib\'{o} regarding the Tutte polynomials of recusively constructible families.

Noy and Rib\'o considered graph sequences obtained from various graph
products. They proved that if $G_{1},G_{2},\ldots,G_{n},\ldots$ is
a recursively constructible sequence of graphs and $H$ is a fixed graph,
then the sequences of graphs obtained by applying the Cartesian product
$\cp{G_{n}}{H}$, the strong product $G_{n}\boxtimes H$, the rooted product, 
the tensor 
product $G_{n}\times H$ and the join $G_{n}+H$ are all recursively
constructible, as are variants of them. 
This implies, for example, that several well known families are recursively
constructible, such as cyclic ladders ($\cp{C_n}{K_2}$) and,
for any fixed $t$, grids ($\cp{P_n}{P_t}$) and complete bipartite graphs ($K_{n,t}$).

E. Fischer and J. A. Makowsky generalized this result to a wide family
of graph polynomials, including in particular the domination polynomial
in \cite{ar:FM08}:
\begin{theorem}
(Corollary of \cite{ar:FM08}) Let $G_{1},G_{2},\ldots,G_{n},\ldots$
be a recursively constructible sequence of graphs
\footnote{Fischer and Makowsky actually use a wider definition, namely 
{\em iteratively constructible} sequences of graphs, which also covers 
e.g. cliques and complete bipartite graphs $K_{n,n}$.  } .
The sequences of domination polynomials of $G_{n}$ satisfy a linear
recurrence relation with coefficients in $\mathbb{Z}[x]$. 
\end{theorem}
Fischer and Makowsky's result is not constructive, in the sense that
their proof does not directly emit the desired recurrence relation.
Rather, it only attests to the existence of a recurrence relation. 

From \cite{ar:FM08} it follows that, for every fixed $i$, the sequence
$d_{|V(G_{n})|-i}(G_{n})$ of dominating sets of size $|V(G_{n})|-i$,
satisfies a linear recurrence relation over $\mathbb{Z}$. Note $|V(G_{n})|=cn+r$
for $c,r\in\mathbb{N}$ which do not depend on $n$. It is natural
to consider the number of dominating sets of size e.g. $d_{\left\lceil qn+p\right\rceil }(G_{n})$,
$p,q\in\mathbb{Q}$. Such dominating sets arise as the sizes of
minimum dominating sets of various graphs. For example, M. S. Jacobson
and L. F. Kinch showed in \cite{ar:JacobsonKinch83} the domination
number of the ladder $L_{n}$ is $\left\lceil \frac{n+1}{2}\right\rceil $.
T. Y. Chang and W. E. Clark showed in \cite{ar:ChangClark} that the
domination number of a $5\times n$ grids is $\left\lfloor \frac{6n+8}{5}\right\rfloor $
for large enough $n$. In fact, \cite{ar:JacobsonKinch83} and \cite{ar:ChangClark}
show that the dominating numbers of grids $t\times n$ for $1\leq t\leq6$
are roundings of numbers of the form $qn+p$. S. Alikhani and Y. Peng
considered the domination polynomials of paths in \cite{ar:AlikhaniPengPaths}.
In particular, they computed sequences of coefficients such as $d_{n}(P_{3n})$
and $d_{n+1}(P_{3n+2})$. Similar sequences of coefficients were considered
by Alikhani and Peng for cycles in \cite{ar:AlikhaniPengCertainII2010}. 

We consider the number of linear sized dominating sets. We prove:
\begin{theorem}\label{th:alg}
 Let $G_{1},G_{2},\ldots,G_{n},\ldots$ be a recursively
constructible sequence of graphs. Then the generating functions of 
\begin{enumerate}
\item $d_{\left\lfloor qn+p\right\rfloor }(G_{n})$ 
\item $d_{\left\lceil qn+p\right\rceil }(G_{n})$ 
\end{enumerate}

are algebraic for every $q,p\in\mathbb{Q}$.
\end{theorem}

In particular, this implies that $d_{\left\lfloor qn+p\right\rfloor }(G_{n})$
and $d_{\left\lceil qn+p\right\rceil }(G_{n})$ are {\em P-recursive}
or {\em holonomic}, i.e. that they satisfy linear recurrence relations
with coefficients which are polynomials in $n$. 

\begin{example}
As a simple example of Theorem \ref{th:alg}, consider the sequence $d_n(\cp{K_n}{K_2})$. 
Theorem \ref{t:krk2} gives the following explicit expression for $D(K_r\Box K_2,x)$:
\[
(x+1)^{2n}-2(x+1)^n+2x^n+1
\]
From this expression we can easily extract $d_n(\cp{K_n}{K_2})$, which is the coefficient of $x^n$. 
In $(x+1)^{2n}$, the coefficient of $x^n$ is $\binom{2n}{n}$ using the binomial expansion of $(x+1)^{2n}$. 
The coefficient of $x^n$ in $-2(x+1)^n +2x^n+1$ is $0$. So,
\[
d_n(\cp{K_n}{K_2}) = \binom{2n}{n}
\]
is the central binomial coefficient, which is well-known to have an algebraic generating function. 
Similar expressions can be obtained for any fixed $r$ from Theorem \ref{t:krks} 

\end{example}

Theorem \ref{th:alg} can be applied to show that the number of {}``linearly
small'' (or {}``linearly large'') dominating sets is P-recursive.

\begin{theorem}
\label{th:small} Let $G_{1},G_{2},\ldots,G_{n},\ldots$ be a recursively
constructible sequence of graphs and let $q\in\mathbb{Q}$. Then the
generating functions of the number of dominating sets of size at most
(at least) $q|V(G_{n})|$ is algebraic. 
\end{theorem}

\subsection{Proofs}

We assume the reader is familiar with the basics of rational and algebraic
generating functions, as described e.g. in \cite{bk:StanleyII-1999}.
The following lemma is folklore:
\begin{lemma}
If $F(x)$ and $G(x)$ are rational generating functions, so is their
Hadamard product $F*G(x)$. 
\end{lemma}

We use the following theorem:
\begin{theorem}
[H. Furstenburg \cite{ar:HF67}\label{th:Furst}] If $F(s,t)$ is a rational generating
function in $s$ and $t$, then ${\rm diag}\, F$ is algebraic. 
\end{theorem}

We can now prove Theorem \ref{th:alg}:
\begin{proof}
[Proof of Theorem \ref{th:alg}] We prove case 1. Case 2 can be
proven analogously. Consider the generating function
\[
F(x,t)=\sum_{n,m=0}^{\infty}d_{n}(G_{m})x^{n+m-(\left\lfloor qm+p\right\rfloor )}t^{m}
\]
We will prove $F(x,t)$ is a rational function in $x$ and $t$. The
diagonal of $F$ is 
\[
{\rm diag}\, F(y)=\sum_{m=0}^{\infty}d_{\left\lfloor qm+p\right\rfloor }(G_{m})y^{m}
\]
and by Theorem \ref{th:Furst}, ${\rm diag}\, F(y)$ is algebraic. 

It remains to prove that $F(x,t)$ is rational. Since the sequence
$D(G_{m},x)$ satisfies a linear recurrence relation with coefficients
in $\mathbb{Z}[x]$, the power series 
\[
H(t)=\sum_{m=0}^{\infty}D(G_{m},x)t^{m}
\]
is rational. Using the definition of $D(G_{m},x)$, $H(t)$ can be
rewritten as a power series over $\mathbb{Q}$ with indeterminates
$x$ and $t$:
\[
H(x,t)=\sum_{n,m=0}^{\infty}d_{n}(G_{m})x^{n}t^{m}
\]
and note that $H(x,t)$ is rational. Substituting $t$ with $xt$
we get a new rational power series:
\[
H_{1}(x,t)=\sum_{n,m=0}^{\infty}d_{n}(G_{m})x^{n+m}t^{m}\,.
\]
Let
\[
R(t)=\sum_{m=0}^{\infty}x^{-\left\lfloor qm+p\right\rfloor }t^{m}\:.
\]
 We have $F(x,t)=H_{1}(t)*R(t)$. Using again the closure property of the Hadamard
product, $F(x,t)$ is rational if $R(t):=\sum_{m=0}^{\infty}r(m,x)t^{m}$ is.
Let $q=\frac{a}{b}$ with $a,b\in\mathbb{Z}$ and $b>0$ so that
$r(m,x)$ satisfies $r(m,x)=x^{-a}r(m-b,x)$, implying that $R(t)$ is rational. 
\end{proof}

We restate Theorem \ref{th:small} as follows:
\begin{theorem}
(Number of small dominating sets) Let $G_{1},G_{2},\ldots,G_{n},\ldots$
be a recursively constructible sequence of graphs. Then the generating
function of 

\[
{\displaystyle \sum_{i=0}^{\left\lfloor q'|V(G_{n})|+p'\right\rfloor }d_{i}(G_{n})}
\]

is algebraic for every $q',p'\in\mathbb{Q}$.

\end{theorem}
\begin{proof}
Let $c$ and $r$ be the natural numbers guaranteed such that $|V(G_{n})|=cn+r$.
Let $q=q'c$ and $p=p' + rq'$. Then we need to consider the generating
function of 
\[
{\displaystyle \sum_{i=0}^{\left\lfloor qn+p\right\rfloor }d_{i}(G_{n})}
\]

In the proof of Theorem \ref{th:alg} we proved $F(x,t)$ is rational.
Consider 
\[
F_{1}(x,t)=F(x,t)\cdot\frac{1}{1-x}\,.
\]
$F_{1}$ is given by 
\[
F_{1}(x,t)=\sum_{n,m=0}^{\infty}\sum_{i=0}^{\infty}d_{n}(G_{m})x^{i+n+m-(\left\lfloor qm+p\right\rfloor )}t^{m}\,.
\]
We have that $i+n+m-\left\lfloor qm+p\right\rfloor =m$ iff $i+n=\left\lfloor qm+p\right\rfloor $.
Hence, 
\[
{\rm diag}\, F(y)=\sum_{m=0}^{\infty}\,\sum_{
\substack{i,n\in\mathbb{N}:\\ i+n=\left\lfloor qm+p\right\rfloor}
}d_{n}(G_{m})y^{m}=\sum_{m=0}^{\infty}\,\sum_{n=0}^{\left\lfloor qm+p\right\rfloor }d_{n}(G_{m})y^{m}\,.
\]
 Again case (2) is similar. 
\end{proof}

\section{An application to complexity}

In this section we show an application of Theorem \ref{t:spgkr} to the
Turing complexity of the domination polynomial. We assume the reader
is familiar with the basics of counting complexity theory as described
e.g. in \cite[Chapter 17]{bk:AroraBarak}.

Many graph polynomials studied in the literature have been shown to
be $\#P$-hard to compute with respect to Turing reductions, and the
domination polynomial is no exception. This follows from the $\#P$-completeness
of the number of dominating sets. This remains true even on restricted
graph classes, see e.g. \cite{pr:KijimaOU11}.

A common further step towards understanding the complexity of a particular
graph polynomial is to assess the hardness of computation of its evaluations.
A classic result of this kind can be found in \cite{ar:JVW1990},
where it is shown that the Tutte polynomial is $\#P$-hard to compute
for any rational evaluation, except those in a semi-algebraic set
of low dimension which are polynomial-time computable. Similar dichotomy
theorems have been shown for the cover polynomial \cite{ar:BlaserDell07},
the interlace polynomial \cite{pr:BlaserHoffmann08} and the edge
elimination polynomial and its specializations \cite{ar:Hoffmann10}. 

For any fixed $\gamma\in\mathbb{Q}$, we denote by $D(-,\gamma)$
the problem of computing for an input graph $G$ the evaluation $D(G,\gamma)$
of the domination polynomial. We limit ourselves to rational evaluations
and remain within the realm of Turing complexity in order to avoid
a discussion of appropriate computation models which is not central
to this paper. 
\begin{theorem}
\label{th:complexity} Given $\gamma\in\mathbb{Q}\backslash\{0,-1,-2\}$,
the graph parameter $D(-,\gamma)$ is $\#P$-hard to compute.
\end{theorem}
\begin{proof}
Let $\gamma\in\mathbb{Q}\backslash\{0,-1,-2\}$. We show an algorithm
which on an input graph $G$ with $n$ vertices computes $D(G,x)$
in polynomial time in $n$ using an oracle to $D(-,\gamma)$. Since
$D(G,x)$ is $\#P$-hard, so is $D(-,\gamma)$. The algorithm is as
follows:
\begin{enumerate}
\item For every $r\in\{1,\ldots,n+1\}$, compute 
$D(G\boxtimes K_{r},\gamma)=D(G,(1+\gamma)^{r}-1)$.
$D(G\boxtimes K_{r},\gamma)$ is computed using the oracle to $D(-,\gamma)$
(and $D(G,(1+\gamma)^{r}-1)$ is obtained by Theorem \ref{t:spgkr}). 
\item Interpolate $D(G,x)$ from the values 
\[
\left(x_{0},D(G,x_{0})\right)=\left((1+\gamma)^{i}-1),D(G,(1+\gamma)^{i}-1)\right)\,,
\]
$i=1,\ldots,n+1$.
\end{enumerate}

$D(G,x)$ can be interpolated from the computed values since the values
$(1+\gamma)^{r}-1$ are pairwise distinct (because $\gamma\not=0,-1,-2$)
and $D(G,x)$ has degree $n$. 

\end{proof}

\section{Conclusion}
In this paper we studied the domination polynomials of families of graphs given by products. 
While our results cover some important families of graphs obtained by products, there remain some open problems which we believe deserve attention.

Our work can be extended by finding analogous formulae to more families:
\begin{problem}
~
\begin{enumerate}
\item
How can Theorem \ref{t:GK2} be extended to deal with basic Cartesian product families such as $\cp{G}{K_s}$, 
$\cp{G}{P_s}$, $\cp{G}{C_s}$, etc.?
\item
Can analogs of Theorem \ref{t:spgkr} be found for $\st{G}{P_s}$, $\st{G}{C_s}$, etc.?
\item
What other families of graphs obtained using graph products have simple explicit formulae in the spirit of Theorem \ref{t:krks}?
\end{enumerate}
\end{problem}

We have not considered families such as square grids $\cp{P_n}{P_n}$:
\begin{problem}
What can be proven about families $G_n \oplus H_n$ for some graph product $\oplus$ of recursively constructible families of graphs, both in general and for important special cases such as grids? 
\end{problem}

Theorem 6.1 leaves open the complexity of two special evaluations of the domination polynomial, which 
are not solved using a product-based reduction. $D(-,0)$ is clearly polynomial-time computable, but the following
remains open:
\begin{problem}
How hard to compute are the graph parameters $D(-,-1)$ and $D(-,-2)$? 

$D(G,-1)$ has been studied from a combinatorial point of view in
\cite{ar:KPT12GCOM}. To our knowledge $D(G,-2)$ has not received
attention in the literature. 
\end{problem}

\section*{Acknowledgement}

We would like to thank Janos Makowsky for a useful discussion which
led to the inclusion of Theorem \ref{th:complexity} in this paper.

\end{document}